\documentclass[12pt]{article}

\usepackage{amsmath, epsfig, cite, setspace}
\usepackage{amssymb}
\usepackage{amsfonts}
\usepackage{latexsym}
\usepackage{amsthm}

\makeatletter
\renewcommand{\@seccntformat}[1]{{\csname the#1\endcsname}.\hspace{.5em}}
\makeatother

\newtheorem{thm}{Theorem}[section]

\newtheorem{lem}[thm]{Lemma}
\newtheorem{remark}[thm]{Remark}

\renewcommand{\qed}{\hfill$\Box$\medskip}
\renewcommand{\thefootnote}{*}

\numberwithin{equation}{section}

\begin{document}
\begin{center}
{\large\bf Proof of two supercongruences of truncated hypergeometric series ${}_4F_3$}
\end{center}

\vskip 2mm \centerline{Guo-Shuai Mao}
\begin{center}
{Department of Mathematics, Nanjing
University of Information Science and Technology, Nanjing 210044,  People's Republic of China\\
{\tt maogsmath@163.com  } }
\end{center}
%%date: January 4, 2011
%\vskip 5mm
%\noindent {\it Suggested Running title}: Two Identities of Gould

\vskip 0.7cm \noindent{\bf Abstract.}
In this paper, we prove two supercongruences conjectured by Z.-W. Sun via the Wilf-Zeilberger method. One of them is, for any prime $p>3$,
\begin{align*}
{}_4F_3\bigg[\begin{matrix} \frac76&\frac12&\frac12&\frac12\\ &\frac16&1&1\end{matrix}\bigg|-\frac18\bigg]_{\frac{p-1}2}\equiv p\left(\frac{-2}p\right)+\frac{p^3}4\left(\frac2p\right)E_{p-3}\pmod{p^4},
\end{align*}
where $\left(\frac{\cdot}p\right)$ stands for the Legendre symbol, and $E_{n}$ is the $n$-th Euler number.



\renewcommand{\thefootnote}{**}

\section{Introduction}
For $n\in\mathbb{N}=\{0,1,2,\ldots\}$, define {\it the truncated hypergeometric function}
$$
{}_{m+1}F_m\bigg[\begin{matrix}
\alpha_0&\alpha_1&\ldots&\alpha_m\\
&\beta_1&\ldots&\beta_m
\end{matrix}\bigg|\,z\bigg]_n:=\sum_{k=0}^{n}\frac{(\alpha_0)_k(\alpha_1)_k\cdots(\alpha_m)_k}{(\beta_1)_k\cdots(\beta_m)_k}\cdot\frac{z^k}{k!},
$$
where $\alpha_0,\ldots,\alpha_m,\beta_1,\ldots,\beta_m,z\in\mathbb{C}$ and
$$
(\alpha)_k=\begin{cases}\alpha(\alpha+1)\cdots(\alpha+k-1) &\text{if }k\geq 1,\\
1 &\text{if }k=0.\end{cases}
$$
In the past decade, many researchers studied supercongruences via the Wilf-Zeilberger (WZ) method (see, for instance, \cite{gl-arxiv-2019,CXH-rama-2016,he-jnt-2015,hm-rama-2017,mz-rama-2019,oz-jmaa-2016,sun-ijm-2012}). Chen, Xie and He \cite{CXH-rama-2016} confirmed a supercongruence conjetured by Z.-W. Sun \cite{sun-scm-2011}, which says that for any prime $p>3$,
$$
{}_4F_3\bigg[\begin{matrix} \frac{4}{3}&\frac12&\frac12&\frac12\\ &\frac1{3}&1&1\end{matrix}\bigg|-8\bigg]_{p-1}\equiv p(-1)^{(p-1)/2}+p^3E_{p-3}\pmod{p^4},
$$
where $\{E_n\}$ are the Euler numbers given by
$$\frac{2e^t}{e^{2t}+1}=\sum_{n=0}^\infty E_n\frac{t^n}{n!}\ (|t|<\frac{\pi}2).$$
For $n\in\mathbb{N}$, define
$$H_n:=\sum_{0<k\leq n}\frac1k,\ H_n^{(2)}:=\sum_{0<k\leq n}\frac1{k^2},\ H_0=H_0^{(2)}=0,$$
where $H_n$ with $n\in\mathbb{N}$ are often called the classical harmonic numbers. Let $p>3$ be a prime. J. Wolstenholme \cite{wolstenholme-qjpam-1862} proved that
\begin{align}\label{hp-1}
H_{p-1}\equiv0\pmod{p^2}\ \mbox{and}\ H_{p-1}^{(2)}\equiv0\pmod p,
\end{align}
which imply that
\begin{align}
\binom{2p-1}{p-1}\equiv1\pmod{p^3}.\label{2p1p}
\end{align}

Z.-W. Sun \cite{sun-ijm-2012} proved the following supercongruence by the WZ method, for any odd prime $p$,
\begin{equation}\label{sun}
{}_4F_3\bigg[\begin{matrix} \frac54&\frac12&\frac12&\frac12\\ &\frac14&1&1\end{matrix}\bigg|-1\bigg]_{p-1}\equiv(-1)^{\frac{(p-1)}2}p+p^3E_{p-3}\pmod{p^4}.
\end{equation}
Guo and Liu \cite{gl-arxiv-2019} showed that for any prime $p>3$,
\begin{equation}\label{glp4}
\sum_{k=0}^{(p+1)/2}(-1)^k(4k-1)\frac{\left(-\frac12\right)_k^3}{(1)_k^3}\equiv p(-1)^{(p+1)/2}+p^3(2-E_{p-3})\pmod{p^4},
\end{equation}
where $(a)_n=a(a+1)\ldots(a+n-1) (n\in\{1,2,\ldots\})$ with $(a)_0=1$ is the raising factorial for $a\in\mathbb{C}$. Guo and his coauthors also studied $q$-analogues of Ramanujan-type supercongruences and $q$-analogues of supercongruences of van Hamme (see, for instance, \cite{g-jmaa-2018,guo-rama-2019,guos-rm-2020}).\\
Long \cite{long-2011-pjm} and Chen, Xie and He \cite{CXH-rama-2016} proved independently that, for any odd prime $p$,
$$
{}_4F_3\bigg[\begin{matrix} \frac76&\frac12&\frac12&\frac12\\ &\frac16&1&1\end{matrix}\bigg|-\frac18\bigg]_{\frac{p-1}2}\equiv p\left(\frac{-2}p\right)\pmod{p^2}.
$$

In this paper, we first obtain the following result which confirms a challenging conjecture of Sun\cite{sun-scm-2011}:
\begin{thm}\label{Thsun} Let $p>3$ be a prime. Then
\begin{equation}\label{6n1-512}
{}_4F_3\bigg[\begin{matrix} \frac76&\frac12&\frac12&\frac12\\ &\frac16&1&1\end{matrix}\bigg|-\frac18\bigg]_{\frac{p-1}2}\equiv p\left(\frac{-2}p\right)+\frac{p^3}4\left(\frac2p\right)E_{p-3}\pmod{p^4}.
\end{equation}
\end{thm}
\begin{remark}\rm This congruence was conjectured by Z.-W. Sun (see, \cite[Conjecture 5.1]{sun-scm-2011} or \cite[(2.22)]{sun-numb-2019} with $n=1$).
\end{remark}
In order to prove Theorem \ref{Thsun}, we prove the following theorem first.
\begin{thm}\label{Thhar} Let $p$ be an odd prime. Then
\begin{equation}\label{hkhk}
\sum_{k=1}^{(p-1)/2}\frac{\binom{2k}k}{2^k}H_k^2\equiv(-1)^{(p-1)/2}q_p(2)^2-E_{p-3}\pmod p,
\end{equation}
where $q_p(2)$ stands for the Fermat quotient $(2^{p-1}-1)/p$.
\end{thm}
\begin{remark}\rm In view of \cite[Remark 1.2]{sun-ijm-2015}, sun said that he was able to show that
$$
\sum_{k=0}^{p-1}\frac{\binom{2k}k}{2^k}H_k^2\equiv\frac12(-1)^{(p-1)/2}q_p(2)^2-\frac12E_{p-3}\pmod p.
$$
But he made a typing error, because by (\ref{hkhk}), we have
$$
\sum_{k=0}^{p-1}\frac{\binom{2k}k}{2^k}H_k^2\equiv(-1)^{(p-1)/2}q_p(2)^2-E_{p-3}\pmod p.
$$
\end{remark}
Recall that the Euler polynomials $\{E_n(x)\}$ are defined by
$$\frac{2e^{xt}}{e^{t}+1}=\sum_{n=0}^\infty E_n(x)\frac{t^n}{n!}\ (|t|<\pi).$$
We also obtain the following result which was also conjectured by Z.-W. Sun:
\begin{thm}\label{Thsun1} Let $p>3$ be a prime. Then
\begin{equation}
{}_4F_3\bigg[\begin{matrix} \frac76&\frac12&\frac12&\frac12\\ &\frac16&1&1\end{matrix}\bigg|-\frac18\bigg]_{p-1}\equiv p\left(\frac{-2}p\right)+\frac{p^3}{16}E_{p-3}\left(\frac14\right)\pmod{p^4}.
\end{equation}
\end{thm}
\begin{remark}\rm This congruence was the conjecture of Z.-W. Sun \cite[(2.16)]{sun-numb-2019} with $n=1$.

\end{remark}
Our main tool in this paper is the WZ method. We shall prove Theorems \ref{Thhar} and \ref{Thsun} in Sections 2 and 3, respectively. And the last Section is devoted to proving Theorem \ref{Thsun1}. We can see that Long \cite{long-2011-pjm} and Chen, Xie and He \cite{CXH-rama-2016} only proved the congruence in Theorem \ref{Thsun} modulo $p^2$, so the congruence in Theorem \ref{Thsun} is intrinsically difficult. And the WZ method is far from enough, we should solve numerous intermediate obstacles. We also empoly the summation package \verb"Sigma"\cite{S} based on difference ring/field algorithms in order to derive some non-trivial harmonic number identities. The arguments are not just $-8$ and $-\frac18$, there is also $\frac14$, and if the items $\frac76$ and $\frac16$ are replaced by other suitable items, the arguments can be $-1$, $\frac1{64}$ and so on, which were studied in other papers. And why the error term is Euler number or Euler polynomial or Bernoulli number or Bernoulli polynomial, that's because these supercongruences are $p$-adic analogues of Ramanujan $\frac1{\pi}$ series, and these series may relate to the $L$ function. We know that the value of the $L$ function at some negative intergers is related to the Bernoulli number with character. The Bernoulli number with character $\left(\frac{-1}p\right)$ is Euler number, the Bernoulli number with character $\left(\frac{-2}p\right)$ is the value of Euler polynomial at $\frac14$, and the Bernoulli number with character $\left(\frac{-3}p\right)$ can be used to represent the value of Bernoulli polynomial at $\frac13$. At last, we want to explain that although the WZ method is used in other papers, but new methods and innovations are needed to solve obstacles in the middle, because the WZ method is a large framework for us to confirm supercongruences similar to the circle method and the sieve method in Analytic Number Theory, of course, it maybe not that powerful.
\section{Proof of Theorem \ref{Thhar}}
\begin{lem}\label{Lemhkhk} For any positive integer $n$, we have
\begin{align}\label{nkhkhk}
\sum_{k=1}^n\binom{n}{k}(-2)^kH_k^2=&(-1)^n\left(\frac{H_n^{(2)}}2+H_n^2-2\sum_{k=1}^n\frac{(-1)^k}{k^2}-2H_n\sum_{k=1}^n\frac{(-1)^k}k\right)\notag\\
&+(-1)^n\left(-\frac{\left(\sum_{k=1}^n\frac{(-1)^k}k\right)^2}2+3\sum_{k=1}^n\frac{(-1)^kH_k}k\right)
\end{align}
and
\begin{align}\label{nkhk}
\sum_{k=1}^n\binom{n}{k}(-2)^kH_k=(-1)^nH_n-(-1)^n\sum_{k=1}^n\frac{(-1)^k}k.
\end{align}
\end{lem}
\begin{proof}
These two identities can be found and proved by summation package \verb"Sigma". For the details of using \verb"Sigma" to find and prove identities, we may refer the readers to \cite{S,MWW,L,M}.
\end{proof}
\begin{lem}\label{Lemhk} Let $p>3$ be a prime. Then
$$
\sum_{k=1}^{(p-1)/2}\frac{(-1)^kH_k}k\equiv\frac12q_p(2)^2+(-1)^{(p-1)/2}E_{p-3}\pmod p.
$$
\end{lem}
\begin{proof}
We know that $(-1)^{k-1}\binom{p-1}{k-1}\equiv1-pH_{k-1}\pmod{p^2}$, so
\begin{align*}
p\sum_{k=1}^{(p-1)/2}\frac{(-1)^k}kH_{k-1}&\equiv\sum_{k=1}^{(p-1)/2}\frac{(-1)^k}k\left(1-(-1)^{k-1}\binom{p-1}{k-1}\right)\\
&=\sum_{k=1}^{(p-1)/2}\frac{(-1)^k}k+\frac1p\sum_{k=1}^{(p-1)/2}\binom{p}k\pmod{p^2}.
\end{align*}
And it is easy to see that
$$
\sum_{k=1}^{(p-1)/2}\binom{p}k=\sum_{k=(p+1)/2}^{p-1}\binom{p}{p-k}=\sum_{k=(p+1)/2}^{p-1}\binom{p}{k}.
$$
Thus,
$$
\sum_{k=1}^{(p-1)/2}\binom{p}k=\frac12\left(\sum_{k=1}^{(p-1)/2}\binom{p}k+\sum_{k=(p+1)/2}^{p-1}\binom{p}{k}\right)=\frac12\sum_{k=1}^{p-1}\binom{p}{k}=2^{p-1}-1.
$$
In view of \cite[Theorem 3.2]{sun-jnt-2008}, we have, for any prime $p>3$
\begin{align}\label{-1k}
\sum_{k=1}^{(p-1)/2}\frac{(-1)^k}k=H_{\lfloor p/4\rfloor}-H_{(p-1)/2}\equiv-q_p(2)+\frac12q_p(2)^2-(-1)^{(p-1)/2}E_{p-3}\pmod p.
\end{align}
Hence
$$
\sum_{k=1}^{(p-1)/2}\frac{(-1)^k}kH_{k-1}\equiv\frac1p\left(\sum_{k=1}^{(p-1)/2}\frac{(-1)^k}k+q_p(2)\right)\equiv\frac12q_p(2)^2-(-1)^{(p-1)/2}E_{p-3}\pmod{p^2}.
$$
By \cite[Theorem 3.5]{sun-jnt-2008}, we know that for any prime $p>5$,
\begin{equation}\label{-1k2}
\sum_{k=1}^{(p-1)/2}\frac{(-1)^k}{k^2}=\frac12H_{\lfloor p/4\rfloor}^{(2)}-H_{(p-1)/2}^{(2)}\equiv2(-1)^{(p-1)/2}E_{p-3}\pmod p.
\end{equation}
Therefore, for any prime $p>5$, we have
$$
\sum_{k=1}^{(p-1)/2}\frac{(-1)^k}kH_{k}\equiv\frac12q_p(2)^2+(-1)^{(p-1)/2}E_{p-3}\pmod p.
$$
The $p=5$ case can be checked directly. Now we finish the proof of Lemma \ref{Lemhk}.
\end{proof}
\noindent{\it Proof of (\ref{hkhk})}. In light of \cite[Theorems 3.2,3.5]{sun-jnt-2008}, we have
$$H_{(p-1)/2}^{(2)}\equiv0\pmod p\ \ \mbox{and}\ \ H_{(p-1)/2}\equiv-2q_p(2)\pmod p.$$
These, with Lemma \ref{Lemhkhk}, (\ref{-1k}), (\ref{-1k2}) and Lemma \ref{Lemhk} yield that
$$\sum_{k=1}^{(p-1)/2}\binom{(p-1)/2}{k}(-2)^kH_k^2\equiv(-1)^{(p-1)/2}q_p(2)^2-E_{p-3}\pmod p.$$
Therefore the proof of (\ref{hkhk}) is completed with $\binom{2k}k\equiv\binom{(p-1)/2}k(-4)^k\pmod p$.\qed
\section{Proof of Theorem \ref{Thsun}}
\qquad We will use the following WZ pair which appears in \cite{CXH-rama-2016} to prove Theorem \ref{Thsun}. For nonnegative integers $n, k$, define
$$
F(n,k)=(-1)^{n+k}\frac{(6n-2k+1)}{2^{9n-3k}}\frac{(2n+2k)!(2n-2k)!\binom{2n-2k}{n-k}}{(n+k)!(n-k)!n!^2}
$$
and
$$
G(n,k)=(-1)^{n+k}\frac{n^2(2n+2k)!(2n-2k)!\binom{2n-2k}{n-k}}{2^{9n-3k-4}(2n+2k-1)(n+k)!(n-k)!n!^2}.
$$
Clearly $F(n,k)=G(n,k)=0$ if $n<k$. It is easy to check that
\begin{equation}\label{FG}
F(n,k-1)-F(n,k)=G(n+1,k)-G(n,k)
\end{equation}
for all nonnegative integer $n$ and $k>0$.

Summing (\ref{FG}) over $n$ from $0$ to $(p-1)/2$ we have
$$
\sum_{n=0}^{(p-1)/2}F(n,k-1)-\sum_{n=0}^{(p-1)/2}F(n,k)=G\left(\frac{p+1}2,k\right)-G(0,k)=G\left(\frac{p+1}2,k\right).
$$
Furthermore, summing both side of the above identity over $k$ from $1$ to $(p-1)/2$, we obtain
\begin{align}\label{wz1}
\sum_{n=0}^{(p-1)/2}F(n,0)=F\left(\frac{p-1}2,\frac{p-1}2\right)+\sum_{k=1}^{(p-1)/2}G\left(\frac{p+1}2,k\right).
\end{align}
\begin{lem}\label{mor}{\rm (\cite{Mor})} For any prime $p>3$, we have
$$
\binom{p-1}{(p-1)/2}\equiv(-1)^{(p-1)/2}4^{p-1}\pmod{p^3}.
$$
\end{lem}
\begin{lem} \label{Fp1p1} Let $p>3$ be a prime. Then
$$F\left(\frac{p-1}2,\frac{p-1}2\right)\equiv (-1)^{(p-1)/2}p\left(1-pq_p(2)+p^2q_p(2)^2\right)\pmod{p^4}.$$
\end{lem}
\begin{proof} In view of the definition of $F(n,k)$, we have
\begin{align*}
F\left(\frac{p-1}2,\frac{p-1}2\right)=\frac{(2p-1)(2p-2)!}{2^{3p-3}(p-1)!\left(\frac{p-1}2\right)!^2}=\frac{(2p-1)\binom{2p-2}{p-1}\binom{p-1}{\frac{p-1}2}}{2^{3p-3}}=\frac{p\binom{2p-1}{p-1}\binom{p-1}{\frac{p-1}2}}{2^{3p-3}}.
\end{align*}
This, with (\ref{2p1p}), Lemma \ref{mor} and $2^{p-1}=1+pq_p(2)$ yields that
$$
F\left(\frac{p-1}2,\frac{p-1}2\right)\equiv(-1)^{(p-1)/2}\frac{p}{2^{p-1}}\equiv(-1)^{(p-1)/2}p\left(1-pq_p(2)+p^2q_p(2)^2\right)\pmod{p^4}.
$$
Therefore the proof of Lemma \ref{Fp1p1} is complete.
\end{proof}
By the definition of $G(n,k)$ we have
\begin{align*}
G(n,k)&=(-1)^{n+k}\frac{n^2(2n+2k)!(2n-2k)!\binom{2n-2k}{n-k}}{2^{9n-4-3k}(2n+2k-1)(n+k)!(n-k)!n!^2}\notag\\
&=(-1)^{n+k}\frac{\binom{2n+2k}{n+k}\binom{2n-2k}{n-k}^2(n+k)!(n-k)!}{2^{9n-4-3k}(2n+2k-1)(n-1)!^2}\notag\\
&=(-1)^{n+k}\frac{\left(\frac12\right)_{n+k}\left(\frac12\right)_{n-k}^2}{(2n+2k-1)(n-1)!^2(n-k)!2^{3n-k-4}}.
\end{align*}
It is easy to see that
$$
\frac{\left(\frac12\right)_{n+k}}{(n+k)!}=\frac{\binom{2n+2k}{n+k}}{4^{n+k}}, \ \ \left(\frac12\right)_{n+k}=\left(\frac12\right)_{n}\left(\frac12+n\right)_{k}
$$
and
$$
\left(\frac12\right)_{n-k}\left(\frac12+n-k\right)_{k-1}=\left(\frac12\right)_{n-1}.
$$
So we have
\begin{align}\label{Gpk}
G(n,k)&=(-1)^{n+k}\frac{\left(\frac12\right)_{n}\left(\frac12+n\right)_{k}\left(\frac12\right)_{n-1}^2}{(2n+2k-1)(n-1)!^2\left(\frac12+n-k\right)_{k-1}^2(n-k)!2^{3n-k-4}}\notag\\
&=\frac{(-1)^{n+k}}{2^{9n-k-8}}\frac{\binom{2n}n\binom{2n-2}{n-1}^2\left(\frac12+n\right)_{k}n!}{(2n+2k-1)\left(\frac12+n-k\right)_{k-1}^2(n-k)!}.
\end{align}
\begin{lem}\label{nk-2hk} For any prime $p>3$, we have
$$
\sum_{k=1}^{(p-1)/2}\binom{(p-1)/2}k(-2)^kH_k\equiv(-1)^{(p-1)/2}\left(-q_p(2)+\frac p2q_p(2)^2\right)+pE_{p-3}\pmod {p^2}.
$$
\end{lem}
\begin{proof}
Set $n=(p-1)/2$ in (\ref{nkhk}), we have
\begin{align*}
\sum_{k=1}^{(p-1)/2}\binom{(p-1)/2}k(-2)^kH_k=(-1)^{(p-1)/2}H_{(p-1)/2}-(-1)^{(p-1)/2}(H_{\lfloor\frac p4\rfloor}-H_{(p-1)/2}),
\end{align*}
by \cite[Theorem 3.2]{sun-jnt-2008} we immediately obtain the desired result.
\end{proof}
\begin{lem} \label{G12} For any primes $p>3$, modulo $p^4$ we have
$$
\sum_{k=1}^{(p-1)/2}G\left(\frac{p+1}2,k\right)\equiv\left(\frac{-2}p\right)p+\frac{p^3}4\left(\frac{2}p\right)E_{p-3}-(-1)^{(p-1)/2}p\left(1-pq_p(2)+p^2q_p(2)^2\right).
$$
\end{lem}
\begin{proof}
By (\ref{Gpk}), we have
\begin{align*}
G\left(\frac{p+1}2,k\right)&=\frac{(-1)^{(p+1)/2}\binom{p+1}{(p+1)/2}\binom{p-1}{(p-1)/2}^2}{2^{(9p+9)/2-8}}\frac{(-2)^k\left(\frac p2+1\right)_{k}\left(\frac{p+1}2+1-k\right)_{k}}{(p+2k)\left(\frac p2+1-k\right)_{k-1}^2}\\
&=\frac{(-1)^{(p-1)/2}2p\binom{p-1}{(p-1)/2}^3}{2^{(9p-9)/2}}\frac{(p/2-k)(-2)^{k-1}\left(\frac p2+1\right)_{k}\left(\frac{p+1}2+1-k\right)_{k-1}}{(p+2k)\left(\frac p2-k\right)_{k}\left(\frac p2+1-k\right)_{k-1}}.
\end{align*}
It is easy to check that
\begin{align*}
\frac{\left(\frac{p}2+1\right)_k}{\left(\frac{p}2-k\right)_k}&\equiv\frac{k!\left(1+\frac{p}2H_k+\frac{p^2}4\sum_{1\leq i<j\leq k}\frac1{ij}\right)}{(-1)^kk!(\left(1-\frac{p}2H_k+\frac{p^2}4\sum_{1\leq i<j\leq k}\frac1{ij}\right)}\\
&\equiv(-1)^k\left(1+pH_k+\frac{p^2}2H_k^2\right)\pmod{p^3},
\end{align*}
and modulo $p^3$ we have
\begin{align*}
\left(\frac p2+1-k\right)_{k-1}&=\left(\frac p2+1-k\right)\ldots\left(\frac p2-1\right)\\
&\equiv(-1)^{k-1}(k-1)!\left(1-\frac p2H_{k-1}+\frac{p^2}8\left(H_{k-1}^2-H_{k-1}^{(2)}\right)\right).
\end{align*}
In view of \cite[(4.4)]{sun-ijm-2015}, we have
\begin{align*}
&(-2)^{k-1}\left(\frac{p+3}2-k\right)_{k-1}=(2k-p-3)(2k-5-p)\ldots(1-p)\\
&\equiv(2k-3)!!\left(1-p\sum_{j=1}^{k-1}\frac1{2j-1}+\frac{p^2}2\left(\sum_{j=1}^{k-1}\frac1{2j-1}\right)^2-\frac{p^2}2\sum_{j=1}^{k-1}\frac1{(2j-1)^2}\right)\\
&\equiv(2k-3)!!\frac{\binom{(p-1)/2}{k-1}(-4)^{k-1}}{\binom{2k-2}{k-1}}\pmod{p^3}.
\end{align*}
So modulo $p^4$, we have
\begin{align*}
G\left(\frac{p+1}2,k\right)&\equiv\frac{(-1)^{(p-1)/2}2p\binom{p-1}{(p-1)/2}^3}{2^{(9p-9)/2}}\frac{k-p/2}{p+2k}\binom{(p-1)/2}{k-1}(-2)^{k-1}\\
&\cdot\left(1+\frac{3p}2H_{k-1}+\frac{9p^2}8H_{k-1}^2+\frac{p^2}8H_{k-1}^{(2)}+\frac pk+\frac{p^2}{2k^2}+\frac{3p^2H_{{k-1}}}{2k}\right)
\end{align*}
Now we have the following congruence modulo $p^4$
\begin{align*}
\sum_{k=1}^{(p-1)/2}G\left(\frac{p+1}2,k\right)&\equiv\frac{(-1)^{(p-1)/2}2p\binom{p-1}{(p-1)/2}^3}{2^{(9p-9)/2}}\sum_{k=1}^{(p-1)/2}\binom{(p-1)/2}{k-1}(-2)^{k-1}\\
&\cdot\left(\frac12+\frac{3p}4H_{k-1}+\frac{9p^2}{16}H_{k-1}^2+\frac{p^2}{16}H_{k-1}^{(2)}\right).
\end{align*}
By Lemma \ref{mor}, \cite[(1.6)]{sun-ijm-2015}, (\ref{hkhk}), Lemma \ref{nk-2hk} and
$$
2^{(p-1)/2}\equiv\left(\frac2p\right)\left(1+\frac p2q_p(2)-\frac{p^2}8q_p(2)^2\right)\pmod {p^3},
$$
we have
\begin{align*}
&\sum_{k=1}^{(p-1)/2}G\left(\frac{p+1}2,k\right)\equiv p(-2)^{(3p-3)/2}\\
&\cdot\left(1-(2)^{\frac{p-1}2}+(-1)^{\frac{p-1}2}\frac{p^2}4E_{p-3}-\frac{3p}2q_p(2)+\frac{15p^2}8q_p(2)^2-\frac{3p}22^{\frac{p-1}2}H_{\frac{p-1}2}-\frac{9p^2}82^{\frac{p-1}2}H_{\frac{p-1}2}^2\right)\\
&\equiv p\left(\frac{-2}p\right)\left(1+\frac{3p}2q_p(2)+\frac{3p^2}8q_p(2)^2\right)-p(-1)^{(p-1)/2}\left(1+2pq_p(2)+p^2q_p(2)^2\right)\\
&+\frac{p^3}4\left(\frac2p\right)E_{p-3}-\frac{3p^2}2\left(\frac{-2}p\right)\left(q_p(2)+\frac{3p}2q_p(2)^2\right)+\frac{15p^3}8\left(\frac{-2}p\right)q_p(2)^2\\
&-\frac{3p^2}2(-1)^{(p-1)/2}(1+2pq_p(2))(-2q_p(2)+pq_p(2)^2)-\frac{9p^3}8(-1)^{(p-1)/2}(-2q_p(2))^2\\
&\equiv p\left(\frac{-2}p\right)+\frac{p^3}4\left(\frac2p\right)E_{p-3}-p(-1)^{(p-1)/2}\left(1-pq_p(2)+p^2q_p(2)^2\right)\pmod{p^4}.
\end{align*}
Therefore we finish the proof of Lemma \ref{G12}.
\end{proof}
\noindent{\it Proof of Theorem \ref{Thsun}}. Combining (\ref{wz1}) with Lemmas \ref{Fp1p1} and \ref{G12}, we immediately obtain that
$$
{}_4F_3\bigg[\begin{matrix} \frac76&\frac12&\frac12&\frac12\\ &\frac16&1&1\end{matrix}\bigg|-\frac18\bigg]_{\frac{p-1}2}=\sum_{n=0}^{(p-1)/2}\frac{6n+1}{(-512)^n}\binom{2n}n^3\equiv p\left(\frac{-2}p\right)+\frac{p^3}4\left(\frac{2}p\right)E_{p-3}\pmod{p^4}.
$$
Therefore the proof of Theorem \ref{Thsun} is finished. \qed
\section{Proof of Theorem \ref{Thsun1}}
By the same WZ pair as in section 3, and summing (\ref{FG}) over $n$ from $0$ to $p-1$, we have
$$
\sum_{n=0}^{p-1}F(n,k-1)-\sum_{n=0}^{p-1}F(n,k)=G(p,k)-G(0,k)=G(p,k).
$$
Furthermore, summing both side of the above identity over $k$ from $1$ to $p-1$, we obtain
\begin{align}\label{wz2}
\sum_{n=0}^{p-1}F(n,0)=F(p-1,p-1)+\sum_{k=1}^{p-1}G(p,k).
\end{align}
\begin{lem} \label{Fp1p12} Let $p>3$ be a prime. Then
$$F(p-1,p-1)\equiv -3p^2\left(1+4p-6pq_p(2)\right)\pmod{p^4}.$$
\end{lem}
\begin{proof} By the definition of $F(n,k)$, we have
\begin{align*}
F(p-1,p-1)&=\frac{(4p-3)(4p-4)!}{2^{6p-6}(2p-2)!{(p-1)!}^2}=\frac{(4p-3)\binom{4p-4}{2p-2}\binom{2p-2}{p-1}}{2^{6p-6}}\\
&=\frac{p\binom{4p-3}{2p-2}\binom{2p-1}{p-1}}{2^{6p-6}}=\frac{p^2\binom{4p-1}{2p-1}\binom{2p-1}{p-1}}{(4p-1)2^{6p-6}}.
\end{align*}
This, with (\ref{2p1p}), $2^{p-1}=1+pq_p(2)$ and
$$
\binom{4p-1}{2p-1}=-\prod_{k=1}^{2p-1}\left(1-\frac{4p}k\right)=3\prod_{k=1}^{p-1}\left(1-\frac{4p}k\right)\cdot\prod_{k=1}^{p-1}\left(1-\frac{4p}{p+k}\right)\equiv3\pmod{p^2}
$$
yields that
$$
F(p-1,p-1)\equiv\frac{3p^2}{(4p-1)2^{6p-6}}\equiv-3p^2\left(1+4p-6pq_p(2)\right)\pmod{p^4}.
$$
Therefore the proof of Lemma \ref{Fp1p12} is complete.
\end{proof}
Recall that the Bernoulli numbers $\{B_n\}$ and the Bernoulli polynomials $\{B_n(x)\}$ are defined as follows:
$$\frac x{e^x-1}=\sum_{n=0}^\infty B_n\frac{x^n}{n!}\ \ (0<|x|<2\pi)\ \mbox{and}\ B_n(x)=\sum_{k=0}^n\binom nkB_kx^{n-k}\ \ (n\in\mathbb{N}).$$
\begin{lem}\label{mos}{\rm [See \cite{MOS}]} Let $x$ and $y$ be variables and $n\in\mathbb{N}$. Then
\begin{align*}
&{\rm (i)}.\  B_{2n+1}=0.\\
&{\rm (ii)}.\ B_n(1-x)=(-1)^nB_n(x).\\
&{\rm (iii)}.\  B_n(x+y)=\sum_{k=0}^n\binom{n}kB_{n-k}(y)x^k.\\
&{\rm (iv)}.\  E_{n-1}(x)=\frac{2^n}n\left(B_n\left(\frac{x+1}2\right)-B_n\left(\frac x2\right)\right).
\end{align*}
\end{lem}
\begin{lem}\label{Lem8k2kk} For any prime $p>3$, we have
$$
\sum_{k=1}^{(p-1)/2}\frac{8^k}{k(2k-1)\binom{2k}k}\equiv\left(\frac{-2}p\right)\sum_{k=1}^{(p-1)/2}\frac{2^k}{k^2\binom{2k}k}\equiv\frac14E_{p-3}\left(\frac14\right)\pmod p.
$$
\end{lem}
\begin{proof} We know that
\begin{align*}
\sum_{k=1}^{(p-1)/2}\frac{8^k}{k(2k-1)\binom{2k}k}&=\sum_{k=1}^{(p-1)/2}\frac{8^{(p+1)/2-k}}{((p+1)/2-k)(p-2k)\binom{p+1-2k}{(p+1)/2-k}}\\
&\equiv\sum_{k=1}^{(p-1)/2}\frac{2^{(p+3)/2}}{8^kk(2k-1)\binom{p-2k}{(p+1)/2-k}}\pmod p.
\end{align*}
It is easy to check that modulo $p$, we have
$$
\binom{p-2k}{\frac{p+1}2-k}\equiv(-1)^{\frac{p+1}2-k}\binom{\frac{p-1}2+k}{2k-1}=\frac{2k(-1)^{\frac{p+1}2-k}}{\frac{p-1}2-k+1}\binom{\frac{p-1}2+k}{2k}\equiv\frac{4k(-1)^{\frac{p-1}2}}{2k-1}\frac{\binom{2k}{k}}{16^k}.
$$
So
$$
\sum_{k=1}^{(p-1)/2}\frac{8^k}{k(2k-1)\binom{2k}k}\equiv(-2)^{(p-1)/2}\sum_{k=1}^{(p-1)/2}\frac{2^k}{k^2\binom{2k}k}\equiv\left(\frac{-2}p\right)\sum_{k=1}^{(p-1)/2}\frac{2^k}{k^2\binom{2k}k}\pmod p.
$$
Now we turn to prove
$$\sum_{k=1}^{(p-1)/2}\frac{2^k}{k^2\binom{2k}k}\equiv\frac14\left(\frac{-2}p\right)E_{p-3}\left(\frac14\right)\pmod p.$$
Setting $n=(p-1)/2$ in \cite[(22)]{mt-jnt-2013}, we have
\begin{align*}
\binom{p-1}{(p-1)/2}\sum_{k=1}^{(p-1)/2}\frac{2^k}{k^2\binom{2k}k}&=\sum_{k=1}^{(p-1)/2}\binom{p-1}{(p-1)/2-k}\frac{v_k(0)}{k^2}+\binom{p-1}{(p-1)/2}\sum_{k=1}^{(p-1)/2}\frac{1}{k^2}\\
&\equiv\sum_{k=1}^{(p-1)/2}(-1)^{(p-1)/2-k}\frac{v_k(0)}{k^2}\pmod p,
\end{align*}
where $v_n(x)$ are defined as follows:
$$
v_0(x)=2,\ \ v_1(x)=x,\ \ \mbox{and}\ \ v_n(x)=xv_{n-1}(x)-v_{n-2}(x)\ \ \mbox{for}\ \ (n>1).
$$
So
\begin{align*}
\sum_{k=1}^{(p-1)/2}\frac{2^k}{k^2\binom{2k}k}&\equiv\sum_{\substack{k=1\\k\equiv2\pmod4}}^{(p-1)/2}(-1)^{k}\frac{-2}{k^2}+\sum_{\substack{k=1\\k\equiv0\pmod4}}^{(p-1)/2}(-1)^{k}\frac{2}{k^2}\\
&\equiv2\left(\sum_{j=1}^{\lfloor\frac{p-1}8\rfloor}\frac1{16j^2}-\sum_{j=1}^{\lfloor\frac{p+3}8\rfloor}\frac1{(4j-2)^2}\right)\pmod p.
\end{align*}
In view of \cite[Lemma2.1]{sun-jnt-2008}, for any $p,m\in\mathbb{N}$ and $k,r\in\mathbb{Z}$ with $k\geq0$, we have
$$
\sum_{\substack{x=0\\{x\equiv r\pmod{m}}}}^{p-1}x^k=\frac{m^k}{k+1}\left(B_{k+1}\left(\frac{p}{m}+\left\{\frac{r-p}{m}\right\}\right)-B_{k+1}\left(\left\{\frac{r}{m}\right\}\right)\right).
$$
This, with (i) and (iii) of Lemma \ref{mos} yields that
\begin{align*}
\sum_{j=1}^{\lfloor\frac{p-1}8\rfloor}\frac1{j^2}=64\sum_{\substack{k=1\\k\equiv0\pmod8}}^{p-1}\frac{1}{k^2}\equiv64\sum_{\substack{x=0\\x\equiv0\pmod8}}^{p-1}x^{p-3}\equiv-\frac12B_{p-2}\left(\left\{-\frac p8\right\}\right)\pmod p
\end{align*}
and
\begin{align*}
\sum_{j=1}^{\lfloor\frac{p+3}8\rfloor}\frac1{(2j-1)^2}&=16\sum_{j=1}^{\lfloor\frac{p+3}8\rfloor}\frac1{(8j-4)^2}=16\sum_{\substack{j=1\\j\equiv4\pmod8}}^{p-1}\frac1{j^2}\equiv16\sum_{\substack{x=0\\j\equiv4\pmod8}}^{p-1}x^{p-3}\\
&\equiv-\frac18B_{p-2}\left(\left\{\frac{4-p}8\right\}\right)\pmod p.
\end{align*}
Hence
$$
\sum_{k=1}^{(p-1)/2}\frac{2^k}{k^2\binom{2k}k}\equiv\frac1{16}\left(B_{p-2}\left(\left\{\frac{4-p}8\right\}\right)-B_{p-2}\left(\left\{-\frac p8\right\}\right)\right)\pmod p.
$$
It is easy to see that
\begin{align*}
B_{p-2}\left(\left\{\frac{4-p}8\right\}\right)-B_{p-2}\left(\left\{-\frac p8\right\}\right)&=\begin{cases}B_{p-2}\left(\left\{\frac{3}8\right\}\right)-B_{p-2}\left(\left\{\frac 78\right\}\right) &\tt{if}\ p\equiv1\pmod {8},\\B_{p-2}\left(\left\{\frac{1}8\right\}\right)-B_{p-2}\left(\left\{\frac 58\right\}\right) &\tt{if}\ p\equiv3\pmod {8},\\ B_{p-2}\left(\left\{\frac{7}8\right\}\right)-B_{p-2}\left(\left\{\frac 38\right\}\right) &\tt{if}\ p\equiv5\pmod {8},\\B_{p-2}\left(\left\{\frac{5}8\right\}\right)-B_{p-2}\left(\left\{\frac 18\right\}\right) &\tt{if}\ p\equiv7\pmod {8}.\end{cases}
\end{align*}
So by (ii) of Lemma \ref{mos} , we immediately get that
$$
B_{p-2}\left(\left\{\frac{4-p}8\right\}\right)-B_{p-2}\left(\left\{-\frac p8\right\}\right)=\left(\frac{-2}p\right)\left(B_{p-2}\left(\left\{\frac{1}8\right\}\right)-B_{p-2}\left(\left\{\frac 58\right\}\right)\right).
$$
Then by (iv) of Lemma \ref{mos}, we have
$$
\sum_{k=1}^{(p-1)/2}\frac{2^k}{k^2\binom{2k}k}\equiv\frac1{16}\left(4\left(\frac{-2}p\right)E_{p-3}\left(\frac14\right)\right)=\frac14\left(\frac{-2}p\right)E_{p-3}\left(\frac14\right)\pmod p.
$$
Therefore the proof of Lemma \ref{Lem8k2kk} is complete.
\end{proof}
\begin{lem} \label{G13} For any primes $p>3$, we have
$$
\sum_{k=1}^{(p-1)/2}G(p,k)\equiv\frac{p^3}{16}E_{p-3}\left(\frac14\right)\pmod{p^4}.
$$
\end{lem}
\begin{proof}
By (\ref{Gpk}), we have
\begin{align*}
G(p,k)&=\frac{(-1)^{k+1}}{2^{9p-k-8}}\frac{\binom{2p}p\binom{2p-2}{p-1}^2\left(\frac12+p\right)_kp!}{(2p+2k-1)\left(\frac12+p-k\right)_{k-1}^2(p-k)!}\\
&\equiv p^3\frac{(-2)^k\left(\frac12\right)_k}{(2k-1)\left(\frac12-k\right)_{k-1}^2(p-k)!}\pmod{p^4}.
\end{align*}
We know that
$$
\frac{\left(\frac12\right)_k}{\left(\frac12-k\right)_{k-1}}=\frac{\frac12\cdot\frac32\ldots\frac{2k-1}2}{\frac{1-2k}2\cdot\frac{3-2k}2\ldots\frac{-3}2}=\frac{(-1)^{k-1}}2,
$$
$$
\left(\frac12-k\right)_{k-1}=(-1)^{k-1}\frac{3\cdot5\ldots(2k-1)}{2^{k-1}}=(-1)^{k-1}\frac{(2k-2)!!}{2^{k-1}}
$$
and
$$
(p-k)!=(p-k)(p-(k+1))\ldots(p-(p-1))\equiv(-1)^{p-k}\frac{(p-1)!}{(k-1)!}\equiv\frac{(-1)^k}{(k-1)!}\pmod p.
$$
Hence
$$
G(p,k)\equiv\frac{p^3}4\frac{8^k}{k(2k-1)\binom{2k}k}\pmod {p^4}.
$$
Therefore we can immediately get the desired result with Lemma \ref{Lem8k2kk}.
\end{proof}
\begin{lem}\label{Lemp+12} Let $p>3$ be a prime. Then
$$
G\left(p,\frac{p+1}2\right)\equiv p\left(\frac{-2}p\right)\left(1-\frac{3p}2q_p(2)+\frac{15p^2}8q_p(2)^2\right)\pmod{p^4}.
$$
\end{lem}
\begin{proof} By (\ref{Gpk}) again and (\ref{2p1p}), we have
\begin{align*}
G\left(p,\frac{p+1}2\right)&=\frac{(-1)^{(p-1)/2}}{2^{9p-8-(p+1)/2}}\frac{\binom{2p}p\binom{2p-2}{p-1}^2\left(\frac12+p\right)_{(p+1)/2}p!}{3p\left(\frac12+p-\frac{p+1}2\right)_{(p-1)/2}^2\left(\frac{p-1}2\right)!}\\
&\equiv\frac{(-1)^{(p-1)/2}p}{2^{9p-8-(p+1)/2}}\frac{\left(\frac12+p\right)_{(p-1)/2}(p-1)!}{\left(\frac p2+1\right)_{(p-1)/2}^2\left(\frac{p-1}2\right)!}\pmod{p^4}.
\end{align*}
It is easy to check that by (\ref{hp-1}) and $H_{(p-1)/2}^{(2)}\equiv0\pmod p$,
\begin{align*}
&\left(\frac12+p\right)_{(p-1)/2}\\
&\equiv\frac{(p-2)!!}{2^{(p-1)/2}}\left(1+p\sum_{k=1}^{(p-1)/2}\frac2{2k-1}+\frac{p^2}2\left(\sum_{k=1}^{(p-1)/2}\frac2{2k-1}\right)^2-\frac{p^2}2\sum_{k=1}^{(p-1)/2}\frac4{(2k-1)^2}\right)\\
&\equiv\frac{(p-1)!}{2^{p-1}\frac{p-1}2!}\left(1-pH_{(p-1)/2}+\frac{p^2}2H_{(p-1)/2}^2\right)\pmod{p^3}
\end{align*}
and
\begin{align*}
\left(\frac p2+1\right)_{(p-1)/2}^2&\equiv\left(\frac{p-1}2!\right)^2\left(1+\frac p2H_{\frac{p-1}2}+\frac{p^2}8\left(H_{\frac{p-1}2}^2-H_{\frac{p-1}2}^{(2)}\right)\right)^2\\
&\equiv\left(\frac{p-1}2!\right)^2\left(1+pH_{\frac{p-1}2}+\frac{p^2}2H_{\frac{p-1}2}^2\right)\pmod{p^3}.
\end{align*}
So
\begin{align}\label{imp}
\frac{\left(\frac12+p\right)_{(p-1)/2}(p-1)!}{\left(\frac p2+1\right)_{(p-1)/2}^2\left(\frac{p-1}2!\right)}&\equiv\frac{(p-1)!^2(1-pH_{(p-1)/2}+\frac{p^2}2H_{(p-1)/2}^2)}{2^{p-1}\left(\frac{p-1}2!\right)^4(1+pH_{(p-1)/2}+\frac{p^2}2H_{(p-1)/2}^2)}\notag\\
&\equiv\frac{\binom{p-1}{(p-1)/2}^2}{2^{p-1}}\left(1-2pH_{\frac{p-1}2}+2p^2H_{\frac{p-1}2}^2\right)\pmod{p^3}.
\end{align}
Hence by $H_{(p-1)/2}\equiv-2q_p(2)+pq_p(2)^2\pmod {p^2}$ and
$$
2^{(11p-11)/2}\equiv\left(\frac2p\right)\left(1+\frac{11p}2q_p(2)+\frac{99p^2}8q_p(2)^2\right)\pmod{p^3},
$$
 we have
\begin{align*}
G\left(p,\frac{p+1}2\right)&\equiv(-1)^{(p-1)/2}p\frac{1+4pq_p(2)+6p^2q_p(2)^2}{2^{(11p-11)/2}}\\
&\equiv p\left(\frac{-2}p\right)\left(1-\frac{3p}2q_p(2)+\frac{15p^2}8q_p(2)^2\right)\pmod{p^4}.
\end{align*}
Therefore the proof of Lemma \ref{Lemp+12} is complete.
\end{proof}
\begin{lem}\label{Lemgpk} For any prime $p>3$, we have
$$
\sum_{k=(p+3)/2}^{p-1}G(p,k)\equiv\left(\frac{-2}p\right)3p^2\left(\frac12q_p(2)-\frac58pq_p(2)^2\right)+3p^2(1+4p-6pq_p(2))\pmod{p^4}.
$$
\end{lem}
\begin{proof}
By (\ref{Gpk}), we have
\begin{align*}
G(p,k)&=\frac{(-1)^{k+1}}{2^{9p-k-8}}\frac{\binom{2p}p\binom{2p-2}{p-1}^2\left(\frac12+p\right)_kp!}{(2p+2k-1)\left(\frac12+p-k\right)_{k-1}^2(p-k)!}\\
&\equiv\frac{(-1)^{k+1}}{2^{9p-k-9}}\frac{p^3\left(\frac12+p\right)_k(p-1)!}{(2p-1)^2(2p+2k-1)\left(\frac12+p-k\right)_{k-1}^2(p-k)!}\\
&=\frac{(-1)^{k+1}}{2^{9p-k-9}}\frac{6p^2\left(\frac12+p\right)_{(p-1)/2}\left(\frac{3p}2+1\right)_{k-(p+1)/2}(p-1)!}{4(2p+2k-1)\left(\frac p2+1\right)_{(p-1)/2}^2\left(\frac p2-k+\frac{p+1}2\right)_{k-(p+1)/2}^2(p-k)!}.
\end{align*}
So modulo $p^4$, we have
\begin{align*}
\sum_{k=(p+3)/2}^{p-1}G(p,k)\equiv-\frac{3p^2\left(\frac12+p\right)_{(p-1)/2}(p-1)!}{2^{9p-8}\left(\frac p2+1\right)_{(p-1)/2}^2}\sum_{k=1}^{(p-3)/2}\frac{(-2)^{k+(p+1)/2}\left(\frac{3p}2+1\right)_{k}}{(3p+2k)\left(\frac p2-k\right)_{k}^2\left(\frac{p-1}2-k\right)!}.
\end{align*}
It is easy to check that
\begin{align*}
&\sum_{k=1}^{(p-3)/2}\frac{(-2)^{k+(p+1)/2}\left(\frac{3p}2+1\right)_{k}}{(3p+2k)\left(\frac p2-k\right)_{k}^2\left(\frac{p-1}2-k\right)!}\equiv\sum_{k=1}^{(p-3)/2}\frac{k!\left(1+\frac{3p}2H_k\right)(-2)^k}{(3p+2k)k!^2\left(1-\frac p2H_k\right)^2\left(\frac{p-1}2-k\right)!}\\
&\equiv\frac1{\left(\frac{p-1}2\right)!}\sum_{k=1}^{(p-3)/2}\binom{(p-1)/2}k(-2)^k\left(1+\frac{5p}2H_k\right)\left(\frac1{2k}-\frac{3p}{4k^2}\right)\\
&\equiv\frac1{\left(\frac{p-1}2\right)!}\sum_{k=1}^{(p-3)/2}\binom{(p-1)/2}k(-2)^k\left(\frac1{2k}-\frac{3p}{4k^2}+\frac{5p}{4k}H_k\right)\pmod{p^2}.
\end{align*}
Hence by (\ref{imp}) and Lemma \ref{mor}, we have
\begin{align*}
&\sum_{k=\frac{p+3}2}^{p-1}G(p,k)\equiv\frac{(-1)^{\frac{p-1}2}3p^2}{2^{9p-8-\frac{p+1}2}}\frac{\binom{p-1}{\frac{p-1}2}^2}{2^{p-1}}(1+4pq_p(2))\sum_{k=1}^{\frac{p-3}2}\binom{\frac{p-1}2}k(-2)^k\left(\frac1{2k}-\frac{3p}{4k^2}+\frac{5p}{4k}H_k\right)\\
&\equiv\frac{(-1)^{\frac{p-1}2}3p^2}{2^{6p-5-\frac{p+1}2}}(1+4pq_p(2))\sum_{k=1}^{\frac{p-3}2}\binom{\frac{p-1}2}k(-2)^k\left(\frac1{2k}-\frac{3p}{4k^2}+\frac{5p}{4k}H_k\right)\pmod{p^4}.
\end{align*}
By the software package {\texttt{Sigma}} we find the following three identities
\begin{align*}
&\sum_{k=1}^n\binom{n}k\frac{(-2)^k}k=-H_{n}+\sum_{k=1}^n\frac{(-1)^k}k,\\
&\sum_{k=1}^n\binom{n}k\frac{(-2)^k}{k^2}=-\frac12H_{n}^{(2)}-\frac12H_n^2+\sum_{k=1}^n\frac{1}k\sum_{j=1}^k\frac{(-1)^j}j,\\
&\sum_{k=1}^n\binom{n}k\frac{(-2)^k}{k}H_k=-\frac12H_{n}^{(2)}-\frac12\left(\sum_{k=1}^n\frac{(-1)^k}k\right)^2+\sum_{k=1}^n\frac{-1)^kH_k}k.
\end{align*}
Set $n=(p-1)/2$ in the above identities, we have
\begin{align*}
\sum_{k=1}^{\frac{p-1}2}\binom{\frac{p-1}2}k\frac{(-2)^k}k=H_{\lfloor\frac p4\rfloor}-2H_{\frac{p-1}2}\equiv q_p(2)-\frac p2q_p(2)^2-p(-1)^{\frac{p-1}2}E_{p-3}\pmod{p^2},
\end{align*}
\begin{align*}
\sum_{k=1}^{\frac{p-1}2}\binom{\frac{p-1}2}k\frac{(-2)^k}{k^2}&=-\frac12H_{\frac{p-1}2}^{(2)}-\frac12H_{\frac{p-1}2}^2+\sum_{j=1}^{\frac{p-1}2}\frac{(-1)^j}j(H_{\frac{p-1}2}-H_{j-1})\\
&\equiv-\sum_{j=1}^{(p-1)/2}\frac{(-1)^j}jH_{j-1}\equiv-\frac12q_p(2)^2+(-1)^{(p-1)/2}E_{p-3}\pmod p,
\end{align*}
\begin{align*}
\sum_{k=1}^{(p-1)/2}\binom{(p-1)/2}k\frac{(-2)^k}{k}H_k\equiv\sum_{j=1}^{(p-1)/2}\frac{(-1)^j}jH_{j}\equiv(-1)^{(p-1)/2}E_{p-3}\pmod p.
\end{align*}
Therefore
\begin{align*}
&\sum_{k=\frac{p+3}2}^{p-1}G(p,k)\equiv\frac{(-1)^{\frac{p-1}2}3p^2}{2^{6p-5-\frac{p+1}2}}(1+4pq_p(2))\left(\frac12q_p(2)+\frac p8q_p(2)^2+(-2)^{\frac{p-1}2}(1+4p-5pq_p(2))\right)\\
&\equiv(-1)^{\frac{p-1}2}3p^2\left(\frac2p\right)\left(\frac{1}2q_p(2)-\frac58pq_p(2)^2\right)+3p^2(1+4p-6pq_p(2))\\
&=3p^2\left(\frac{-2}p\right)\left(\frac{1}2q_p(2)-\frac58pq_p(2)^2\right)+3p^2(1+4p-6pq_p(2))\pmod{p^4}.
\end{align*}
Now the proof of Lemma \ref{Lemgpk} is complete.
\end{proof}
\noindent{\it Proof of Theorem \ref{Thsun1}}. Combining (\ref{wz2}), Lemmas \ref{Fp1p12}, \ref{G13}-\ref{Lemgpk}, we immediately obtain the desired result
$${}_4F_3\bigg[\begin{matrix} \frac76&\frac12&\frac12&\frac12\\ &\frac16&1&1\end{matrix}\bigg|-\frac18\bigg]_{p-1}=\sum_{n=0}^{p-1}\frac{6n+1}{(-512)^n}\binom{2n}n^3\equiv p\left(\frac{-2}p\right)+\frac{p^3}{16}E_{p-3}\left(\frac14\right)\pmod{p^4}.$$
So we finish the proof of Theorem \ref{Thsun1}.\qed

\vskip 3mm \noindent{\bf Acknowledgments.}
The author is funded by the National Natural Science Foundation of China (No. 12001288).

\end{document}